\documentclass[preprint, review,authoryear,12pt]{elsarticle}

\usepackage{amssymb, amsthm, amsmath}
\usepackage{dsfont,multirow}
\providecommand{\abs}[1]{\left\lvert#1\right\rvert}

\newtheorem{theorem}{Theorem}[section]
\newtheorem{lemma}[theorem]{Lemma}

\begin{document}

\begin{frontmatter}



\title{Boundary crossing Random Walks, clinical trials and multinomial sequential estimation.\tnoteref{1}}
\tnotetext[1]{This work is partially supported by PRIN 2008}

\author{Enrico Bibbona\corref{label2}}
\author{Alessandro Rubba}
\address{Dipartimento di Matematica ``G.Peano''\\
Universit\`a di Torino, Italy\\}
\cortext[label2]{enrico.bibbona@unito.it}

\begin{abstract}
A sufficient condition for the uniqueness of multinomial sequential unbiased estimators is provided generalizing a classical result for binomial samples. Unbiased estimators are applied to infer the parameters of multidimensional or multinomial Random Walks which are observed until they reach a boundary. An application to clinical trials is presented.

\end{abstract}


\end{frontmatter}


\section{Introduction} 

In many applications stochastic processes are used to model the behavior of some phenomena up to the first crossing of a threshold. It is the case of neuronal modeling, population dynamics, ruin probabilities... (just to mention a few).
Parametric inference is needed to calibrate such models in order to obtain good fits with experimental data and specific sequential statistical methods are needed (cf. e.g. \cite{susanne}).
In many cases Random Walks (RWs) might be used as toy models for such phenomena.
In the special case where the increments are independent Bernoulli random variables, then a classical result in binomial sequential estimation (cf. \cite{Girshick1946}) may be applied to find an unbiased estimator. In \cite{savage} (updating other references quoted therein) a sufficient condition for the uniqueness of the unbiased estimator is found.
If we have a RW on a higher dimensional lattice or any other RW whose increments have $k$ possible outcomes with probabilities $p_{1}\cdots p_{k}$, a generalization of the above result still applies. Indeed in \cite{multinomial} and \cite{kremers} unbiased sequential estimation is extended to the multinomial context. In such a case a sufficient condition for the uniqueness of the unbiased estimators is not available. The present letter fills this gap and presents a few examples where unbiased estimation is applied to multidimensional or multinomial  boundary crossing RWs.
An application of sequential estimation of the multinomial probabilities that deserve a special attention is that following phase II multistage clinical trials (cf. \cite{multinomial-trials}) where patients are classified according to their respondence to a treatment. A short account of such application concludes the paper. Further relevant results related to the main topic can be found in \cite{bhat-efficient} regarding efficient multinomial sampling plans, in \cite{sinha} for a review of the binomial case and in \cite{sinha2} for generalizations to the quasi-binomial context.

\section{Unbiased multinomial sequential estimation} \label{teoria}

We consider a repeated experiment having $k$ possible outcomes occurring with probabilities $p_{1}\cdots p_{k}$. Denote by $X_{n}=(x_{n}^{1},\cdots,x_{n}^{k})$ the process whose components $x^{i}_{n}\in \mathbb{N}$ count how many occurrences of events of type $i$ we had at the $n$-th (independent) repetition.
The process $X_{n}$ lives in the hyper-plane where the sum of the coordinates is $n$. Denoting by $S_{n}\subset \mathbb{R}^{k}$ the portion of such plane where all the coordinates are positive or null and $S_{n}^{\mathbb{N}}$ the set of points in $S_{n}$ with natural coordinates, for any $n$ we have $X_{n}\in S_{n}^{\mathbb{N}}$.

Let $X_{n}$ be observed until it reaches the boundary $B$ of an accessible region $R\subset\mathbb{N}^{k}$ (we mean those points which are not in $R$ but that might be reached in one step starting from $R$). 

For every point $y\in B$ with coordinates $(y_{1},\cdots, y_{k})$ let us denote by $k(y)$ the number of paths in $R$ that start at the origin and end in $y$ and by $k^{\ast}_{i}(y)$ the number of those that end in $y$ but start in the point whose $i-th$ coordinate is 1 and the others are 0.
The probability that the first hitting to the boundary occurs in $y$ is
\begin{equation}
\mathbb{P}(y)=k(y) p_{1}^{y_{1}}\cdots p_{k}^{y_{k}}.\label{fpt_dens}
\end{equation}
The region $R$ is defined to be \emph{closed} if $\sum_{y\in B} \mathbb{P}(y)=1$. 

\begin{theorem}[\cite{multinomial}]
For any closed region $R$, the ratios
\begin{equation}\hat{p}_{i}(y)=\frac{k_{i}^{\ast}(y)}{k(y)}\label{est}\end{equation}
are unbiased estimators for the probabilities $p_{i}$.\end{theorem}

A sufficient condition on the region R for the estimator \eqref{est} to be the unique bounded unbiased estimator for the binomial (k=2) probability  is given in \cite{savage}.
We are going to generalize it to the multinomial context.
For any $n$ the region $R\in \mathbb{N}^{k}$ and its boundary $B$  project onto $S_{n}^{\mathbb{N}}$ defining the accessible points \emph{of order $n$}, $R_{n}=R\cap S_{n}^{\mathbb{N}}$, the inaccessible points $S_{n}^{\mathbb{N}}-R_{n}$ and (among them) the boundary points $B_{n}=B\cap S_{n}^{\mathbb{N}}$.  $R$ is said to be a \emph{simple} region if for any $n$ the convex hull $H(R_{n})$ of $R_{n}$ does not contain inaccessible points.

\begin{theorem}
If  the region $R\subset\mathbb{N}^{k}$ is simple and closed, the estimators \eqref{est} are the unique bounded unbiased estimators of the parameters $p_{i}$.\label{teorema}
\end{theorem}

We adapt the method in \cite{savage}, but we need the following Lemma (obvious when $k=2$) that will be proved after the main theorem.

\begin{lemma}\label{lemma}
Let R be a simple region,  and $n$ an order such that in $S_{n}^{\mathbb{N}}$ there are both accessible and boundary points.
Among any collection of boundary points $C_{n}\subset B_{n}$  it is always possible to choose a $\bar{y}\in C_{n}$ and a $(k-2)$-hyperplane $\pi_{\bar{y}}$ lying in the $(k-1)$-hyperplane that contains $S_{n}$ such that
\begin{enumerate}
\item $\bar{y}\in\pi_{\bar{y}}$
\item $\pi_{\bar{y}}$ is identified by two linear equations
\begin{equation}\begin{cases}
L(x)=m_{1}x_{1}+\cdots +m_{k}x_{k}= b\\
x_{1}+\cdots+ x_{n}= n
\end{cases}\label{sistema}\end{equation}
where $m_{i}\in \mathbb{N}$ one vanishing and at least one non-vanishing and $b\in \mathbb{N}$.
\item on $R_{n}$ we have $L(x)\geq b+1$
\item at any other boundary point $y\in C_{n}$, we have $L(y)\geq b+1$
\end{enumerate}
\end{lemma}

\begin{proof}[Proof of Theorem \ref{teorema}]
If the theorem were false we would have another unbiased estimator $\hat{U}$ of $p_{i}$   and the difference $\Delta=\hat{p}_{i}-\hat{U}$ would be a non-identically vanishing unbiased estimate of zero. Since the first boundary point $y$ hit by the process is a sufficient statistics (cf. \cite{ferguson}, Section 7.3, Lemma 1), we restrict to those estimators that are function of it and $\mathbb{E}(\Delta)=\sum_{y\in B} \Delta(y) \mathbb{P}(y)=0$. Let $m$ be the smallest integer such that $\Delta$ is not vanishing at one element of $B_m$.
If $R_{m}=\emptyset$  for such $m$ then the region $R$ is finite and the thesis follows from Theorem 4 in \cite{kremers}.
If instead $S_{m}^{\mathbb{N}}$ contains accessible points  we apply Lemma \ref{lemma} to the collection $C_{m}$ of boundary points $y\in B_m$ such that $\Delta(y)\neq 0$ and find a point $\bar{y}$ and a linear combination $L(x)=m_{2}x_{2}+\cdots+m_{k}x_{k}$ (for notational convenience we stipulate that the vanishing coefficient is the first one) with $m_{i}\in \mathbb{N}$ such that $L(\bar{y})=b$ and  that for any $z\in C_{m}\cup R_{m}$ we have $L(z)\geq b+1$. A fortiori $L(y)\geq b+1$ at any $y$ in any $B_{n}$ with $n>m$ since any such a $y$ may only be reached evolving from an $x \in R_{m}$. 
\noindent For some positive $\Delta^{\ast}$ we have
\begin{equation}\abs{\Delta(\bar{y})}\,k(\bar{y})\, p_{1}^{\bar{y}_1}\cdots p_{k}^{\bar{y}_{k}}= \abs{\sum_{y:\big\{\substack{L(y)\geq b+1\\ \Delta(y)\neq 0}} \hspace{-4mm}\Delta(y)\,\mathbb{P}(y)}\leq \Delta^{\ast} \hspace{-4mm} \sum_{y:\big\{\substack{L(y)\geq b+1\\ \Delta(y)\neq 0}} \hspace{-4mm}\mathbb{P}(y).\label{ineq}\end{equation}

We are going to show that there are values of the parameters at which such inequality cannot hold.
By construction any path from the origin to an $y\in B$ such that $\Delta(y)\neq0$ and $L(y)\geq b+1$ either ends in $C_{m}$ or crosses $R_{m}$. In $R_{m}\cup C_{m}$ we have a finite number  $F$ of points $z^{1}\cdots z^{F}$ and there $L(z^{i})\geq b+1$. For any $y\in B$ such that $\Delta(y)\neq0$ and $L(y)\geq b+1$ we have 
\begin{equation}\mathbb{P}(y)=  \mathbb{P}(y\,|R_{m}\cup C_{m})\mathbb{P}(R_{m}\cup C_{m})=\mathbb{P}(y\,|R_{m}\cup C_{m})\sum_{s=1}^{F}k(z^{s}) p_{1}^{z_{1}^{s}}\cdots p_{k}^{z_{k}^{s}}\label{Pa}\end{equation}

Let us now choose the parameters $p_{2}\cdots p_{k}$ in such a way that for some common factor $0<p<1$ we have $p_{i}=p^{m_{i}}$ for any $i=2\cdots k$. We get 
\[\mathbb{P}(y)\leq\mathbb{P}(y\,|R_{m}\cup C_{m})\;p^{b+1}\sum_{s=1}^{F}k(z^{s}) \]
and inequality \eqref{ineq} becomes
 \[p^{b}\abs{\Delta(\bar{y})}\,k(\bar{y})\, p_{1}^{\bar{y}_{1}}\leq \Delta^{\ast}p^{b+1} \hspace{-1mm}\sum_{s=1}^{F}k(z^{s})\cdot \hspace{-7mm}\sum_{\quad y:\big\{\substack{L(y)\geq b+1\\ \Delta(y)\neq 0}} \hspace{-4mm}\mathbb{P}(y|R_{m}\cup C_m)\leq \Delta^{\ast}p^{b+1} \hspace{-1mm} \sum_{s=1}^{F}k(z^{s})\]
that is always violated when $p$ is small enough.
\end{proof}

\begin{proof}[Proof of Lemma \ref{lemma}]
Existence of an $y'$  and  of a $\pi_{y'}$ satisfying conditions 1. and 3. with rational coefficients in \eqref{sistema} is ensured by the Separating Hyperplane theorem (cf. \cite{ferguson}, Sec. 2.7) and the density of $\mathbb{Q}$ in $\mathbb{R}$.
To get natural coefficients in \eqref{sistema} it is then sufficient to multiply the first equation by a suitable integer and to add to it the second equation a sufficient number of times. Let us denote by $L'(x)=b'$ the new equation of $\pi_{y'}$ meeting the first three conditions. Condition 4. may still not be fulfilled by $\pi_{y'}$. Let us denote by $c\leq b'$ the minimum value taken by $L'$ on $C_{n}$ and let us consider the plane $\pi_{c}$ with first equation $L'(x)=c$. If it intersects $C_{n}$ in one and only one point we have found both the point and the plane satisfying condition 4. If $C_{n}\cap\pi_{c}$ contains more than one point, let us select one with the following algorithm. Start with the last coordinate $x_{n}$ and select the points in $C_{n}\cap\pi_{c}$ where $x_{k}$ is largest. Among them choose those at which $x_{k-1}$ is largest and continue until the choice of the largest $j$-th coordinate singles out one and only one point $\bar{y}$ of $C_{n}\cap\pi_{c}$.
Now consider the plane $\pi_{\bar{y},r}$ with first equation
\begin{equation} L_{r}(x)=L'(x)-\frac{1}{r}x_{1}-\frac{1}{r^{2}}x_{2}-\cdots \frac{1}{r^{k}}x_{k} = c - \frac{1}{r}\bar{y}_{1}-\frac{1}{r^{2}}\bar{y}_{2}-\cdots- \frac{1}{r^{k}}\bar{y}_{k}=b_{r}.\label{www}\end{equation}

Of course $\pi_{\bar{y},r}$ still passes through $\bar{y}$, and equation \eqref{www}, once multiplied by $r^{k}$, has integer coefficients. Moreover, since $R_{n}$ is finite and since $L(x)-b>0$ for any $x\in R_{n}$, we can take $r$ large enough to ensure both that $L_{r}(x)-b_{r}>0$ for every $x\in R_{n}$ and that the coefficients are natural. The same argument applies to the points in $C_{n}-\pi_{c}$.
Moreover for any $y\in C_{n}\cap \pi_{c}$ we have
\[L_{r}(y)-b_{r}= \frac{1}{r}(\bar{y}_{1}-y_{1})+\frac{1}{r^{2}}(\bar{y}_{2}-y_{2})+\cdots +\frac{1}{r^{k}}(\bar{y}_{k}-y_{k})\]
which is certainly positive due to the algorithm we used to select $\bar{y}$.
\end{proof}

\section{Examples}\label{RW}
In the following examples we derive the unbiased estimators for some multidimensional or multinomial RWs observed up to the crossing of a boundary. 
\subsection{RWs on a bidimensional lattice}
Let $W_{i}$ be a RW on $\mathbb{Z}^{2}$ such that $W_{0}=0$ and $W_{i}=W_{i-1}+ I_{i}$ where the increments $I_{i}$ take the values (0,1), (1,0),(0,-1) and (-1,0) with probabilities $p_{1}, p_{2}, p_{3}$ and $1-\sum_{i=1}^{3}p_{i}$. Let $W_{i}$ be observed up to the first time its second component equals $b>0$.
The process $X_{n}=(x_{n}^{1},\cdots,x_{n}^{4})$ whose components $x^{i}_{n}$ count how many occurrences of increments of type $i$ we had at the $n$-th step of the RW is of the kind described in Section \ref{teoria} and it is observed until it hits $B=\{x\in\mathbb{N}^{4}: x_{1}-x_{3}=b\}$. The accessible region is closed whenever $p_{1}\geq p_{3}>0$ and simple. 
The maximum likelihood (ML) estimators of the $p_{i}$ are $X^{i}_{N}/N$, while the unique unbiased estimators \eqref{est} are 
\[ \hat{p}_{1}= \frac{b-1}{b}\cdot\frac{X^{1}_{N}}{N-1},  \quad \hat{p}_{2}=  \frac{X^{2}_{N}}{N-1}, \quad \hat{p}_{3}=  \frac{b+1}{b}\cdot\frac{X_{N}^{3}}{N-1}\]
The trajectory count is based on the reflection principle (cf. \cite{feller}).

\begin{table}
\begin{center}
\vspace{-5mm}
\begin{tabular}{ccccccc}
&\multicolumn{3}{c}{ML estimators}  &\multicolumn{3}{c}{Unbiased estimators}\\
&mean&sd&m.s.e.&mean&std&m.s.e.\\\hline
$p_{1}=0.4$&0.436&0.081&0.0078&0.400&0.080&0.0063\\
$p_{2}=0.15$&0.148&0.045&0.0020&0.150&0.046&0.0020\\
$p_{3}=0.3$&0.268&0.078&0.007&0.200&0.087&0.008\\\hline
$p_{1}=0.7$&0.727&0.123&0.016&0.701&0.130&0.017\\
$p_{2}=0.1$&0.095&0.072&0.005&0.101&0.077&0.006\\
$p_{3}=0.1$&0.084&0.085&0.007&0.098&0.098&0.010\\
\end{tabular}
\end{center}
\vspace{-4mm}
\caption{Results of inference on a simulated sample of RWs on a bidimensional lattice stopped as soon as their second component reaches the threshold value $b=10$.\label{bidimensionale}}
\end{table}

The results of a simulation study performed on 10.000 paths are shown in Table \ref{bidimensionale}. The performances of the two methods are not much different and the best choice depends on the parameter range. When $p_{1}$ is close to $p_{3}$ some of the unbiased estimators have a smaller mean square error than the corresponding ML, while when $p_{1}$ is higher ML estimates are better.
Let us remark that the estimates of parameters $p_{2}$ and $p_{4}$, in the direction on which the RW is not constrained, are estimated much better than the other two. 

\subsection{A simple RW allowing for null steps}

Let $W_{i}$ be a RW on $\mathbb{Z}$ such that $W_{0}=0$ and $W_{i}=W_{i-1}+ I_{i}$ where the increments $I_{i}$ are 1, 0 or -1 with probabilities $p_{1}, p_{2}$ and $1-\sum_{i=1}^{2}p_{i}$.
Still we count the increments by $X_{n}=(x_{n}^{1},\cdots,x_{n}^{3})$.
$W_{i}$ is observed up to the first time it equals $b>0$ and $X_{i}$ until $X_{1}-X_{3}=b$.
The accessible region is simple and whenever $p_{1}\geq p_{3}>0$ also closed. ML estimators are again the sample proportions, and the unbiased ones are

\[ \hat{p}_{1}= \frac{b-1}{b}\cdot\frac{X^{1}_{N}}{N-1}  \qquad \hat{p}_{2}=  \frac{X^{2}_{N}}{N-1}.\]

\section{Sequential multinomial estimation and clinical trials}
In a multinomial multistage phase II cancer trial (cf. \cite{multinomial-trials}) a group of patients is treated with a new drug and then classified as \emph{responders} if tumor shrinkage is more than 50\%, \emph{non-responders} if it is less and \emph{early progressions} if they undergo a progress in the disease. A decision is taken whether to stop the trial and conclude that the therapy is \emph{promising} (or \emph{ineffective}) if the responders are more (less) than a predetermined value and the early progressions are less (more) than another value. In the intermediate case when the number of respondent patients or of the early progressions is between the thresholds a new group of patients is enrolled and the trial continue to a next stage.
Estimation of the probability of response and early progressions after such trials matters in practice. In the case of a binomial trial (patients are either responders or non-responders) the presence of a bias from ML was already noticed in \cite{JungTrials} and  unbiased estimators were studied. The design proposed in \cite{multinomial-trials} was the following: let $K$ be the maximum number of stages allowed and $n_{s}$ for $s =1 \cdots K$ the number of patients enrolled in each stage. We denote by $N_{s}=\sum_{i\leq s}n_{i}$ the number of patients involved up to the $s$-th stage. The process $X_{j}=(r_{j},j-r_{j}-e_{j}, e_{j})$ counts the number of respondent, non-respondent and early progressions among the first $j$ patients. For any $j\neq N_{s}$ the trial is continued, but when $j=N_{s}$ for some $s<K$ there are three options:
\begin{enumerate}
\item the trial is stopped and the therapy considered \emph{promising} if $r_{N_{s}}\geq \rho^{P}_{s}$ and $e_{N_{s}}\leq \epsilon^{P}_{s}$ and such stopping region is denoted by $B_{N_{s}}^{P}$
\item the trial is stopped and the therapy considered \emph{ineffective} if $r_{N_{s}}\leq \rho^{I}_{s}$ and $e_{N_{s}}\geq \epsilon^{I}_{s}$ and such stopping region is denoted by $B_{N_{s}}^{I}$

\item the trial is continued to stage $s+1$ if $\rho^{I}_{s} \leq r_{N_{s}}\leq \rho^{P}_{s}$ or $\epsilon^{P}_{s}\leq e_{N_{s}}\leq \epsilon^{I}_{s}$ and such continuation region is denoted by $R_{N_{s}}$.
\end{enumerate}

The trial  ends at a random stage $S\leq K$ with a final observation $X_{N_{S}}=(r,N_{S}-r-e,e)$. The probabilities of response and of an early progression can be estimated by means of the unbiased estimators \eqref{est} that are
\begin{align*}\hat{p}_{1}(r,N_{S}-r-e,e)&=\frac{\sum_{R_{N_{1}}}\hspace{-1.6mm}\sum_{R_{N_{2}}}\hspace{-2.6mm}\cdots\hspace{-1mm}\sum_{R_{N_{S-1}}}\hspace{-1.6mm}\binom{n_{1}-1}{r_{N_{1}}-1,y_{1},e_{N_{1}}}\binom{n_{2}}{r_{N_{2}},y_{2},e_{N_{2}}}\cdots\binom{n_{S}}{r_{N_{S}},y_{S},e_{N_{S}}}}{\sum_{R_{N_{1}}}\hspace{-1.6mm}\sum_{R_{N_{2}}}\hspace{-2.6mm}\cdots\hspace{-1mm}\sum_{R_{N_{S-1}}}\hspace{-1.6mm}\binom{n_{1}}{r_{N_{1}},y_{1},e_{N_{1}}}\binom{n_{2}}{r_{N_{2}},y_{2},e_{N_{2}}}\cdots\binom{n_{S}}{r_{N_{S}},y_{S},e_{N_{S}}}}\\
\hat{p}_{3}(r,N_{S}-r-e,e)&=\frac{\sum_{R_{N_{1}}}\hspace{-1.6mm}\sum_{R_{N_{2}}}\hspace{-2.6mm}\cdots\hspace{-1mm}\sum_{R_{N_{S-1}}}\hspace{-1.6mm}\binom{n_{1}-1}{r_{N_{1}},y_{1},e_{N_{1}-1}}\binom{n_{2}}{r_{N_{2}},y_{2},e_{N_{2}}}\cdots\binom{n_{S}}{r_{N_{S}},y_{S},e_{N_{S}}}}{\sum_{R_{N_{1}}}\hspace{-1.6mm}\sum_{R_{N_{2}}}\hspace{-2.6mm}\cdots\hspace{-1mm}\sum_{R_{N_{S-1}}}\hspace{-1.6mm}\binom{n_{1}}{r_{N_{1}},y_{1},e_{N_{1}}}\binom{n_{2}}{r_{N_{2}},y_{2},e_{N_{2}}}\cdots\binom{n_{S}}{r_{N_{S}},y_{S},e_{N_{S}}}}
\end{align*}
where $\binom{n}{r,y,e}$ denotes the multinomial coefficient $\frac{n!}{r! y!e!}$ and the sums are performed over the triples $(r_{N_{i}},y_{i},e_{N_{i}})$ belonging to the continuation regions $R_{N_{i}}$ with $i<S$.

\section{Conclusion}
The main result of the paper is to prove that simplicity of the accessible region $R$ is a sufficient condition to ensure the uniqueness of the unbiased estimators \eqref{est}. Of course the availability (and the uniqueness) of unbiased estimators does not mean that they are the best way to estimate the parameters and the simulation study performed on RWs in Sec.\ref{RW} shows that there are both parameter ranges where the unbiased estimators are superior than ML and vice-versa. The bias of the ML estimators, moreover, can be reduced as in \cite{whitehead} or by bootstrapping and the best method to be used needs to be decided case by case. Multinomial clinical trials provide an important application of the method presented.


\bibliographystyle{elsarticle-harv}
\bibliography{BLSS.bib}

\begin{thebibliography}{13}
\expandafter\ifx\csname natexlab\endcsname\relax\def\natexlab#1{#1}\fi
\expandafter\ifx\csname url\endcsname\relax
  \def\url#1{\texttt{#1}}\fi
\expandafter\ifx\csname urlprefix\endcsname\relax\def\urlprefix{URL }\fi

\bibitem[{Bhat and Kulkarni(1966)}]{bhat-efficient}
Bhat, B.~R., Kulkarni, N.~V., 1966. On efficient multinomial estimation. J.
  Roy. Statist. Soc. Ser. B 28, 45--52.

\bibitem[{Bibbona and Ditlevsen(2010)}]{susanne}
Bibbona, E., Ditlevsen, S., 2010. Estimation in discretely observed markov
  processes killed at a threshold. arXiv: 1011.1356.

\bibitem[{Feller(1971)}]{feller}
Feller, W., 1971. An introduction to probability theory and its applications.
  {V}ol. {II}. Second edition. John Wiley \& Sons Inc., New York.

\bibitem[{Ferguson(1967)}]{ferguson}
Ferguson, T.~S., 1967. Mathematical statistics: {A} decision theoretic
  approach. Probability and Mathematical Statistics, Vol. 1. Academic Press,
  New York.

\bibitem[{Girshick et~al.(1946)Girshick, Mosteller, and Savage}]{Girshick1946}
Girshick, M.~A., Mosteller, F., Savage, L.~J., 1946. Unbiased estimates for
  certain binomial sampling problems with applications. Ann. Math. Statistics
  17, 13--23.

\bibitem[{Jung and Kim({2004})}]{JungTrials}
Jung, S., Kim, K., {2004}. {On the estimation of the binomial probability in
  multistage clinical trials}. {Statistics in Medicine} {23}~({6}), {881--896}.

\bibitem[{Koike(1993)}]{multinomial}
Koike, K.-i., 1993. Unbiased estimation for sequential multinomial sampling
  plans. Sequential Anal. 12~(3-4), 253--259.

\bibitem[{Kremers(1990)}]{kremers}
Kremers, W.~K., 1990. Completeness and unbiased estimation in sequential
  multinomial sampling. Sequential Anal. 9~(1), 43--58.

\bibitem[{Savage(1947)}]{savage}
Savage, L.~J., 1947. A uniqueness theorem for unbiased sequential binomial
  estimation. Ann. Math. Statistics 18, 295--297.

\bibitem[{Sinha et~al.(2008)Sinha, Das, and Mukhoti}]{sinha2}
Sinha, B.~K., Das, K.~K., Mukhoti, S.~K., 2008. On some aspects of unbiased
  estimation of parameters in quasi-binomial distributions. Comm. Statist.
  Theory Methods 37~(18-20), 3023--3028.

\bibitem[{Sinha and Sinha(1992)}]{sinha}
Sinha, B.~K., Sinha, B.~K., 1992. Unbiased sequential binomial estimation. In:
  Current issues in statistical inference: essays in honor of {D}. {B}asu.
  Vol.~17 of IMS Lecture Notes Monogr. Ser. pp. 75--85.

\bibitem[{Whitehead(1986)}]{whitehead}
Whitehead, J., 1986. On the bias of maximum likelihood estimation following a
  sequential test. Biometrika 73~(3), 573--581.

\bibitem[{Zee et~al.({1999})Zee, Melnychuk, Dancey, and
  Eisenhauer}]{multinomial-trials}
Zee, B., Melnychuk, D., Dancey, J., Eisenhauer, E., {1999}. {Multinomial phase
  II cancer trials incorporating response and early progression}. {J Biopharm
  Stat} {9}~({9}), {351--363}.

\end{thebibliography}







\end{document}